\documentclass[11pt]{article}
\usepackage[T1]{fontenc}
\usepackage[margin=1in]{geometry}
\usepackage{verbatim}
\usepackage{mathtools}
\usepackage{dsfont}
\usepackage{amsmath, amssymb, amsfonts, amsthm}
\usepackage[numbers]{natbib}
\usepackage{hyperref}
\hypersetup{
  colorlinks=true,
  linkcolor=NavyBlue,  
  citecolor=NavyBlue,  
  urlcolor=NavyBlue   
}
\usepackage[nameinlink]{cleveref}   
\usepackage{tikz}
\usetikzlibrary{decorations.pathreplacing}

\usepackage{url}
\usepackage{nicefrac}
\usepackage{graphicx}
\usepackage{doi}
\usepackage[dvipsnames]{xcolor}
\usepackage{bbm}
\usepackage[shortlabels]{enumitem}
\usepackage[leqno,fleqn,intlimits]{empheq}
\usepackage{mathrsfs}
\usepackage{soul}

\theoremstyle{plain}
\newtheorem{theorem}{Theorem}[section]
\newtheorem{lemma}[theorem]{Lemma}
\newtheorem{proposition}[theorem]{Proposition}
\newtheorem{corollary}[theorem]{Corollary}
\newtheorem{claim}[theorem]{Claim}

\theoremstyle{definition}
\newtheorem{definition}[theorem]{Definition}

\theoremstyle{remark}
\newtheorem{remark}[theorem]{Remark}

\newtheorem*{notation*}{Notation}

\numberwithin{equation}{section}

\newcommand{\R}{\mathbb{R}}
\newcommand{\Z}{\mathbb{Z}}
\newcommand{\N}{\mathbb{N}}

\newcommand*{\diff}{\mathop{}\!\mathrm{d}}
\newcommand{\wt}{\widetilde}
\newcommand{\1}{\mathbbm 1}

\renewcommand{\P }{\mathbb P}
\renewcommand{\epsilon}{\varepsilon}
\newcommand{\E}{{\mathbb E}}

\DeclareMathOperator*{\Bin}{Bin}
\DeclareMathOperator*{\Po}{Poi}
\DeclareMathOperator{\Ber}{Ber}

\newcommand{\floor}[1]{\left\lfloor #1 \right\rfloor}
\newcommand{\ceil}[1]{\left\lceil #1 \right\rceil}
\newcommand{\prs}[1]{\left( #1 \right)}
\newcommand{\crb}[1]{\left\{ #1 \right\}}

\newcommand{\sqb}[1]{\left[ #1 \right]}

\title{Limiting behaviour of pattern counts in biased binary strings}
\author{Jon V.\@~Kogan\footnote{Einstein Institute of Mathematics, The Hebrew University of Jerusalem, Givat Ram, Jerusalem 91904, Israel. E-mail:jonatan.kogan@mail.huji.ac.il} \  
and Nicolò Paviato\footnote{Einstein Institute of Mathematics, The Hebrew University of Jerusalem, Givat Ram, Jerusalem 91904, Israel. E-mail:  nicolo.paviato@gmail.com}}
\date{\today}

\begin{document}

\maketitle

\begin{abstract}
For $p \in (0,1)$, sample a binary sequence from the infinite product measure of Bernoulli$(p)$ distributions. It is known that for $p=1/2$, almost every binary sequence is Poisson generic in the sense of Peres and Weiss, a property that reflects a specific statistical pattern in the frequency of finite substrings.
However, this behaviour is highly exceptional: it fails for any $p \ne 1/2$. 
In these other cases, we show that the frequency of substrings of almost every sequence has either trivial or peculiar behaviour.
Nevertheless, the Poisson limiting regime can be recovered if one restricts attention to substrings with a fixed number of successes in the Bernoulli$(p)$ trials.
\end{abstract}

\maketitle

\section{Introduction}
What properties does a typical element $x\in\{0,1\}^\N$ satisfy? The simplest answer to this question is \textit{normality}, which Borel~\cite{Bor09} introduced more than a hundred years ago. Assuming $p\in(0,1)$ and that every digit of $x$ is chosen independently with a Bernoulli$(p)$ distribution, so $\P(1)=p$, we say that $x$ is $p$-normal if any finite word $\omega\in\{0,1\}^k$ appears in $x$ with asymptotic frequency $p^{|\omega|}(1-p)^{k-|\omega|}$, where $|\omega|$ is the Hamming weight  of $\omega$, that is the number of coordinates equal to one. Denoting with $\Ber(p)$ the Bernoulli$(p)$ probability measure,  it follows from the ergodic theorem that $\Ber(p)^\N$-almost every (a.e.) sequence is $p$-normal.

Given $x\in\{0,1\}^\N$, a natural follow up question is:  choosing $\omega\in\{0,1\}^k$ at random as above, how often does the word $\omega$ appear in the sequence $x$? If every digit of $x$ and $\omega$ is chosen independently uniformly and $k\ge1$ is large, then the answer is almost surely given by the Poisson distribution. To be more precise, we recall the concept of \textit{Poisson genericity} introduced by Zeev Rudnick~(see~\cite[Definition 1]{AlvBecMer23}). Consider $\{N_k\}_{k\ge1}$ a sequence of positive integers.
For a fixed $x\in\{0,1\}^\N$, we define
\begin{equation}\label{eq:M_k^x}
M_k^x(\omega)=\#\{1\le j\le N_k:(x_j,\dots,x_{j+k-1})=\omega\}, 
\end{equation}
the number of occurrences of $\omega$ in $x$, up to $N_k$.
Letting $N_k=2^k$ and picking $\omega\in\{0,1\}^k$ uniformly, such an~$x$ is said to be simply Poisson generic if $M_k^x$ converges in distribution to a Poisson random variable with parameter one (in short $M_k^x\xrightarrow[]{d}\Po(1)$). 
That is,
\begin{equation*}\label{eq:conv_to_Poi}
\lim_{k\to\infty}\P(M_k^x=n)=\frac{1}{e\cdot n!},  
\end{equation*}
for every $n\in\N\cup\{0\}$. 
Following the notation of~\cite{HocPav25}, we sometimes omit the term “simply” and refer to such $x$ as \textit{Poisson normal} for short.
Note that unqualified the term Poisson generic has a stronger meaning in~\cite{AlvBecMer23}. 

In an unpublished work, Peres and Weiss~\cite[lecture]{Wei20} proved that Poisson normality is strictly stronger than normality, and that $\Ber(1/2)^\N$-a.e.\@ sequence is Poisson generic~\cite{AlvBecMer23}. In the recent work~\cite{AlvBecCesMerPerWei24}, this result was extended to settings with infinite alphabets and exponentially mixing probability measures. It is relevant to mention that, in spite of the abundance of Poisson generic sequences, finding explicit examples is not trivial and the matter was solved for larger alphabets only in 2023 by~\cite{BecSac23}.

A natural generalisation of this setting is to consider different probability measures.
The first result in this direction is~\cite{HocPav25}, where the authors prove that almost sure Poisson normality is still satisfied if sequences are sampled by a non-stationary product measure which is sufficiently close to $\Ber(1/2)^\N$. Here, every digit of the words $\omega\in\{0,1\}^k$ is still sampled independently with probability $p=1/2$. On the other hand, the authors find a threshold for product measures on $\{0,1\}^\N$ above which almost sure Poisson genericity can break down. 

In this paper we consider new measures for both sequences and words, choosing the digits independently with a parameter $p\in(0,1)\setminus\{1/2\}$.
Without loss of generality, we henceforth assume $p> 1/2$. 
All results apply to the $p<1/2$ case by switching the labels of $0$ and $1$, corresponding to a switch of $p$ and $1-p$ in the formulae.
\par By choosing an appropriate asymptotic class for $N_k$, if the digits of~$\omega$ are not equiprobable, we show that $M_k^x$  exhibits a partial escape of mass to infinity.
To state this first result, we  require some extra notation.
We let $H\colon(0,1)\to\R$ denote the binary entropy function
$$H(x)=-x\log_2x-(1-x)\log_2(1-x).$$
Denote with $\Phi\colon\R\to(0,1)$ the cumulative distribution function of a standard Gaussian $\mathcal{N}(0,1)$, so for $s\in\R$,
\begin{equation}\label{eq:cumulant_standard_normal}
\Phi(s)=\frac{1}{\sqrt{2\pi}}\int_{-\infty}^se^{-x^2/2}\diff x.
\end{equation}
See Section~\ref{sec:setup} for a brief review of standard asymptotic notation ($o$, $\omega$, and $\Theta$).
\begin{theorem}\label{thm:main_teo}
Let $p\in(1/2,1)$. Then, as $k\to\infty$ and for $\Ber(p)^\N$-a.e.\@ sequence $x\in\{0,1\}^\N$:
\begin{enumerate}[(1)]
\item\label{item:conv_0_quenched} If ${N}_k=o\bigl(2^{k\cdot H(p)}a^{\sqrt{k}}\bigr)$ for all $a>0$, then $\P(M_k^x=0)\to1$.

\item\label{item:vague_0_quenched} If ${N}_k=\omega\bigl(2^{k\cdot H(p)}a^{\sqrt{k}}\bigr)$ for all $a>0$,  then $ \P( M_k^x\ge n)\to 1$ for any $n\ge0$.

\item\label{item:partial_escape_quenched} If $ N_k=\Theta\bigl( 2^{k\cdot H(p)}a^{\sqrt{k}}\bigr)$ for some $a>0$, then,
\begin{itemize}
\item$\lim_k\P({M}_k^x=0)=\Phi\bigl(-(\log_{p/(1-p)}(a))(p(1-p))^{-1/2}\bigr)$, and

\item$\lim_k\P( M_k^x= n)=0$ for any $n\ge1$.
\end{itemize}
\end{enumerate}
\end{theorem}

Theorem~\ref{thm:main_teo} provides a complete characterization of the typical asymptotic behaviour of the distribution of $M_k^x$. As a comparison, recall that when $p = 1/2$ and $N_k = 2^{k \cdot H(p)} = 2^k$, it holds that $M_k^x \xrightarrow[]{d} \Po(1)$ for $\Ber(p)^\N$-a.e.\@~sequence~$x$~\cite{Wei20}.
Although the loss of mass becomes smaller as $p \to 1/2$, the asymptotic behaviour remains qualitatively different to Poisson. 

The lower order term $a^{\sqrt{k}}$ influences the extent of loss of mass, which increases monotonically with $a$. The exponent $\sqrt{k}$ is a result of an application of the central limit theorem to the Hamming weight of a random word (which is a Bernoulli random variable), see Remark~\ref{rem:application_CLT} and the proof of Proposition~\ref{prop:annealed_convergence_non_intersecting}.
In particular when $a = 1$, exactly half of the total probability mass concentrates at zero (independently of the value of $p$), while the other half escapes to infinity.

As a consequence of Theorem~\ref{thm:main_teo}, we show that when $p\neq 1/2$,  $M_k^x$ typically does not converge to a Poisson random variable.
In other words, Poisson normality is a concept strictly tied to equiprobability of digits.

\begin{corollary}\label{cor:non_Poisson}
Let $\{N_k\}_{k\ge1}$ be any sequence of positive integers and let $p\neq1/2$. Then, for $\Ber(p)^\N$-a.e.\@ $x\in\{0,1\}^\N$, $M_k^x$ does not converge in distribution to  any non-zero random variable.
\end{corollary}

In our last result, for any $p\in(0,1)$ we recover Poisson normality under an artificial equiprobability condition. Choosing uniformly from words $\omega\in\{0,1\}^k$ with a fixed Hamming weight $|\omega|$, we show that $\Ber(p)^\N$-a.e.\@ sequence displays a convergence to Poisson.
For $c\in\R$, we let  $n_k=\lfloor pk-c\sqrt{k}\rfloor$, $k\ge1$. 
We define the set
\[
F_k=\bigl\{\omega\in\{0,1\}^k:|\omega|=n_k\bigr\},
\]
and write $\ddot M^x_{k}\colon\{0,1\}^k\to\N\cup\{0\}$ for the sequence $M^x_k$ from~\eqref{eq:M_k^x}, where $\omega$ is chosen uniformly from $F_k$. The next result shows that this is sufficient to recover the Poisson behaviour.

\begin{theorem}\label{thm:quenched_conditional_poisson}
Let $p\in (0,1)$, $\lambda>0$, and let $N_k=\lfloor \lambda/(p^{n_k}(1-p)^{k-n_k})\rfloor$. Then, 
$\ddot M^x_k\xrightarrow{d} \Po(\lambda)$ for $\Ber(p)^\N$-a.e.\@~$x\in\{0,1\}^\N$.
\end{theorem}

Theorem~\ref{thm:quenched_conditional_poisson} yields two key observations. First, letting $a=(p/(1-p))^c$, a calculation yields that $N_k=\Theta(2^{k\cdot H(p)}a^{\sqrt{k}})$, which is the asymptotic class mentioned in Theorem~\ref{thm:main_teo}-\ref{item:partial_escape_quenched}.
Second, this shows that while the Poisson regime can still be recovered, doing so requires a condition on the Hamming weight, whose probability tends to 0.

\begin{remark}
Theorem~\ref{thm:quenched_conditional_poisson} is stated for  convergence to Poisson for simplicity, but it also holds for a stronger notion of Poisson genericity, similar to the one found in~\cite{AlvBecMer23}. The two differences between this notion and the one of~\cite{AlvBecMer23} are that (a) we consider typical sequences from the measure $\Ber(p)^\N$, $p\in(0,1)$, and (b)~here substrings are chosen uniformly from $F_k$. 
This stronger version of Theorem~\ref{thm:quenched_conditional_poisson} can be proven following the same ideas of~\cite{AlvBecMer23}.
\end{remark}

The rest of the paper is organised as follows. 
In Section~\ref{sec:setup} we introduce the main notation and the two models (intersecting and non-intersecting) that we use. These are used in Section~\ref{sec:annealed} to prove Theorem~\ref{thm:main_teo} on the enlarged probability space $\{0,1\}^\N\times\{0,1\}^k$, also known as annealed case. In Section~\ref{sec:cond_Poi_conv} we borrow ideas from~\cite{AlvBecMer23} to prove the annealed version of Theorem~\ref{thm:quenched_conditional_poisson}. Finally, we prove our main theorems in Section~\ref{sec:quenched}, by passing the results of Sections~\ref{sec:annealed} and~\ref{sec:cond_Poi_conv} to the original probability spaces.

\section{Setup and notation}\label{sec:setup}
We let $\mathbb{N}=\{1,2,3,\ldots\}$, and let $p\in(1/2,1)$. We write $\Ber(p)=p\delta_1 + (1-p)\delta_0$, and define the product measures
\begin{equation*}\label{eq:def_measures}
\Ber(p)^k=\prod_{i=1}^k \Ber(p),
\qquad\text{and}\qquad
\Ber(p)^\N=\prod_{i=1}^\infty \Ber(p),
\end{equation*}
on $\{0,1\}^k$ and $\{0,1\}^\N$ respectively. For brevity, we will mostly omit the dependency of the measures on $p$.

\begin{definition}\label{def:HammingWeight}
    Define the \textit{Hamming weight} of a string $\omega\in\{0,1\}^k$ to be $|\omega|=\sum_{j=1}^k\omega_j$.
\end{definition}
\noindent Notice that for any word $\omega\in\{0,1\}^k$, 
\begin{equation}\label{eq:mu^k(omega)}
\Ber^k(\omega)=p^{|\omega|}(1-p)^{k-|\omega|}.    
\end{equation}

 For $c\in \R$ and  $n_k=\lfloor pk-c\sqrt{k}\rfloor$, $k\ge1$, we define the set
\begin{equation}\label{eq:F_k}
    F_k=\bigl\{\omega\in\{0,1\}^k:|\omega|=n_k\bigr\},
\end{equation}
and on the set $\{0,1\}^k$ the probability measure 
\begin{equation}\label{eq:nu}
\nu^k(\cdot)=\Ber^k(\cdot|F_k).    
\end{equation}

In the remaining part of this section, we introduce two models in which both the sequence and the word result from random selection. In the first (non-intersecting) model, we consider independent samples from $\{0,1\}^k$ drawn according to the product measure $\Ber^k$. In the second (intersecting) model, all finite strings are extracted from a single $\Ber^\N$-random sequence $x\in\{0,1\}^\N$. In the latter case, the samples are no longer independent when they have digits in common.

\begin{notation*}
For $x\in\R$, we define the lower and upper integer parts of $x$ as
\[
\floor{x}=\sup\bigl\{n\in\Z:n\le x\bigr\}
\qquad\text{and}\qquad
\ceil{x}=\inf\bigl\{n\in\Z:n\ge x\bigr\}.
\]

\noindent For two sequences $a_n, b_n>0$, we write:
\begin{itemize}
\item $a_n\sim b_n$ if $\lim_{n\to\infty}a_n/b_n=1$,
\item $a_n=o(b_n)$ if $\lim_{n\to\infty}a_n/b_n=0$;
\item $a_n=\omega(b_n)$ if $\lim_{n\to\infty}a_n/b_n=\infty$;
\item $a_n=\Theta(b_n)$ if there are constants $C_1,C_2>0$ such that $C_1\le a_n/b_n\le C_2$, for all sufficiently large $n$.
\item $a_n=O(b_n)$ if there is $C>0$ such that $a_n\le C b_n$ for all sufficiently large $n$.
\end{itemize}
\end{notation*}

\subsection{Non-intersecting model}\label{subsec:non_intersecting}
We let $\{X^k_{i}\}_{k,i\ge0}$ be a family of iid  random variables with $\Ber(p)$ distribution, defined on a common space with probability measure $\P$.
For $k,j\ge1$, we  define
\[
W=W(k)=(X_0^k,\dots,X_{k-1}^k)\qquad\text{and}\qquad
Z^{(j)}=Z^{(j)}(k)=(X^k_{jk},\dots,X^k_{jk+k-1}).
\]
Note that the real random variable $|W|$ follows a Binomial distribution with parameters~$k$ and~$p$ (in brief $|W|\sim\Bin(k,p)$). 

Since, for every $k \ge 1$, the vectors $Z^{(j)}$ are constructed so that they do not share any common entries, we refer to this as the \textit{non-intersecting model}. In this setting, the random variables are independent, as the presence or absence of a given word in one block has no influence on its appearance in another.  Moreover, the random vectors have the common distribution $\P(Z^{(1)}=\omega)=\Ber^k(\omega)$ from~\eqref{eq:mu^k(omega)}, for $\omega\in\{0,1\}^k$. 

For a sequence of positive integers $(\wt{N}_k)_{k\ge1}$, we define
\begin{equation}\label{eq:wtM_k}\textstyle
    \wt{M}_k=\#\bigl\{1\le j\le \wt N_k:W=Z^{(j)}\bigr\}=
    \sum_{j=1}^{\wt N_k}\1_{\{W=Z^{(j)}\}}.
\end{equation}
For every $j\ge1$
\begin{equation}\label{eq:W=Z}
\P(W=Z^{(j)}|W=\omega)=\P(Z^{(j)}=\omega)=p^{|\omega|}(1-p)^{k-|\omega|},
\end{equation}
as in~\eqref{eq:mu^k(omega)}.
Hence, for every $j,k\ge1$,
\begin{equation}\label{eq:exp_Ij}
\begin{split}\textstyle
     \E\bigl[\1_{\{W=Z^{(j)}\}}\bigr]=&\textstyle\sum_{\omega\in\{0,1\}^k}
     \P\bigl(\1_{\{W=Z^{(j)}\}}=1|W=\omega\bigr)\P(W=\omega)\\
     =&\textstyle \sum_{\omega\in\{0,1\}^k}p^{2|\omega|}(1-p)^{2(k-|\omega|)}\\
     =&\textstyle \sum_{i=0}^k\binom{k}{i}p^{2i}(1-p)^{2(k-i)}=(p^2+(1-p)^2)^k,
\end{split}
\end{equation}
by the binomial theorem.
Moreover, for every $n=0,\dots,\wt{N}_k$,
\begin{equation}\label{eq:distribution_wtM}
\begin{split}
\P(\wt{M}_k=n)&\textstyle=\sum_{\omega\in\{0,1\}^k}
\P(\wt{M}_k=n|W=\omega)\P(W=\omega)\\
&\textstyle=\binom{\wt{N}_k}{n}\sum_{\omega\in\{0,1\}^k} \Ber^k(\omega)^{n}(1-\Ber^k(\omega))^{\wt{N}_k-n}\Ber^k(\omega)\\
&\textstyle=\binom{\wt{N}_k}{n}\sum_{i=0}^k\binom{k}{i}
(p^{i}(1-p)^{k-i})^{n+1}(1-p^{i}(1-p)^{k-i})^{\wt{N}_k-n}.
\end{split}
\end{equation}

Given a sequence $\beta(k)\in\R$ for $k\ge1$, we define two sequences of events:
\begin{equation}\label{eq:wtG_k_and_wtH_k}
\begin{split}
\wt G_k(\beta)&=\{|W|\le pk+\beta(k)\}
\qquad\text{and}\qquad
\wt H_k(\beta)=\{|W|> pk+\beta(k)\}.     
\end{split}   
\end{equation}
For brevity, we will sometimes omit the dependency of the events on $\beta$.

Let $\Phi_p\colon\R\to(0,1)$ be the cumulative distribution function of a Gaussian $\mathcal{N}(0,p(1-p))$. Using our notation in~\eqref{eq:cumulant_standard_normal}, we have for $s\in\R$,
\begin{equation}\label{eq:cumulant_p}
\Phi_p(s)=\Phi\bigl((p(1-p))^{-1/2}s\bigr).
\end{equation}

\begin{remark}\label{rem:application_CLT}
If $\beta(k)=s\sqrt{k}$ for some $s\in\R$, the central limit theorem (CLT) yields that $\P(\wt G_k(\beta))\to\Phi_p(s)$ and $\P(\wt H_k(\beta))\to1-\Phi_p(s)$.
Such limits are still valid if a lower order term $b_k=o(\sqrt{k})$ is added to $\beta(k)$. This can be shown by proving that
\[
\P\biggl(\frac{|W|-pk}{\sqrt{k}}\le c\biggr)-
\P\biggl(\frac{|W|-pk}{\sqrt{k}}\le c+o(1)\biggr)\to 0,
\]
 as $k\to\infty$.
\end{remark}

\subsection{Intersecting model}\label{subsec:intersecting}
For $k\ge1$, we denote by $P_k=\Ber^\N\times\Ber^k$  the probability measure defined on $\Omega_k=\{0,1\}^\N\times\{0,1\}^k$, and write $\E_k$ for the corresponding expectation. 
For $j,k\ge1$, we define the indicators $I_{j}:\Omega_k\rightarrow\{0,1\}$
by
\begin{equation}\label{eq:Indicator}
    I_{j}(x,\omega)=\left\{ \begin{array}{cc}
    1 & x_{j}\ldots x_{j+k-1}=\omega,\\
    0 & \text{otherwise}.
    \end{array}\right.
\end{equation}

In this model, the random variables are not independent, since the appearance (or absence) of a word $\omega$ at a given position in $x$ influences the probability of its occurrence in blocks with common entries. Unlike the non-intersecting model described in Subsection~\ref{subsec:non_intersecting}, here all relevant words are extracted from the same random sequence $x$. For this reason, we refer to this model as the \textit{intersecting} one.

We let $\{N_k\}_{k\ge1}$ be a sequence of positive integers, and define the sequence of random variables $M_{k}:(\Omega_k,P_k)\rightarrow\N\cup\{0\}$
by 
\begin{equation}\label{eq:Mk}\textstyle
    M_{k}(x,\omega)=\#\{1\leq j\leq N_k: x_{i}\ldots x_{i+k-1}=\omega\}=\sum_{j=1}^{N_k}I_{j}(x,\omega).
\end{equation}
For $\nu^k$ as in~\eqref{eq:nu}, let $\ddot P_k = \Ber^\N \times \nu^k$ be a probability measure on $\Omega_k$.  We denote by $\ddot M_k$ the sequence $M_k$, when the underlying probability space is $(\Omega_k,\ddot P_k)$.
When notationally useful, we will identify the set $F_k$ with its lift on $\Omega_k$, that is
$\{0,1\}^\N\times F_k$. So, we can write $\ddot P_k(\cdot)=P_k(\cdot\mid F_k)$.

Given a a sequence $\beta(k)\in\R$, we define the following family of  sets:
\begin{equation}\label{eq:G_k_and_H_k}
\begin{split}
G_k(\beta)&=\bigl\{\omega\in\{0,1\}^k:|\omega|\le pk+\beta(k)\bigr\},\\
H_k(\beta)&=\bigl\{\omega\in\{0,1\}^k:|\omega|> pk+\beta(k)\bigr\}.    
\end{split}
\end{equation}
As with the events $\widetilde{G}_k$ and $\widetilde{H}_k$ defined in~\eqref{eq:wtG_k_and_wtH_k}, we will sometimes omit the dependency on $\beta$.
For the random word $W$ defined as in Subsection~\ref{subsec:non_intersecting}, we observe that 
\[
\widetilde{G}_k = \{ |W| \in G_k \} \quad \text{and} \quad \widetilde{H}_k = \{ |W| \in H_k \}.
\]
Note that $\wt G_k$ and $\wt H_k$ are events in the general probability space from Subsection~\ref{subsec:non_intersecting}, while $G_k$ and $H_k$ are subsets of $\{0,1\}^k$.
As done for $F_k$, we will sometimes identify the sets $G_k$ and $H_k$ with respectively $\{0,1\}^\N\times G_k$ and $\{0,1\}^\N\times H_k$.
Furthermore, the following identities hold:
\begin{equation}\label{eq:identities}
P_k(G_k) = \mathbb{P}(\widetilde{G}_k) \qquad \text{and} \qquad P_k(H_k) = \mathbb{P}(\widetilde{H}_k).    
\end{equation}

\begin{remark}\label{remark:application_CLT}
Consider  $\beta(k)=s\sqrt{k}+b_k$, for $s\in\R$ and $b_k=o(\sqrt{k})$ as in Remark~\ref{rem:application_CLT}. Let $G_k(\beta)$ and $H_k(\beta)$ be the sets from~\eqref{eq:G_k_and_H_k}. Using the identities~\eqref{eq:identities}, it follows that $P_k(G_k)\to\Phi_p(s)$ and $P_k(H_k)\to1-\Phi_p(s)$, as $k\to\infty$.   
\end{remark}

\section{Two annealed results}\label{sec:annealed}

We begin this section by proving the necessary convergence results in the non-intersecting setting. These results will then be used to derive the corresponding statements for the intersecting model, which are presented in the next proposition. We let $M_k$ be the sequence of random variables defined in~\eqref{eq:Mk}, and let  $\Phi\colon \R\to(0,1)$ be the cumulative distribution function from~\eqref{eq:cumulant_standard_normal}.

\begin{proposition}\label{prop:convergence_annealed}
Let $p\in(1/2,1)$. Then, as $k\to\infty$:
\begin{enumerate}[(1)]

\item\label{item:conv_0_annealed} If ${N}_k=o(2^{k\cdot H(p)}a^{\sqrt{k}})$ for all $a>0$, then $P_k(M_k=0)\to1$.

\item\label{item:vague_0_annealed} If ${N}_k=\omega(2^{k\cdot H(p)}a^{\sqrt{k}})$ for all $a>0$,  then $ P_k( M_k\ge n)\to 1$ for any $n\ge0$.

\item\label{item:partial_escape_annealed} If $ N_k=\Theta( 2^{k\cdot H(p)}a^{\sqrt{k}})$ for some $a>0$, then,
\begin{itemize}
\item\label{item:partial_escape_1_intersec}$\lim_kP_k({M}_k=0)=\Phi\bigl(-(\log_{p/(1-p)}(a))(p(1-p))^{-1/2}\bigr)$, and
    \item\label{item:partial_escape_2_intersec}$\lim_kP_k( M_k= n)=0$ for any $n\ge1$.
\end{itemize}
\end{enumerate}
\end{proposition}
\noindent This is often referred to as the \textit{annealed} version of Theorem~\ref{thm:main_teo}, as~$M_k$ is defined on a coupled probability space. Meanwhile, the \textit{quenched} case corresponds to a $\Ber^\N$-almost sure result, which in our setting is exactly Theorem~\ref{thm:main_teo}, which is proven in Section~\ref{sec:quenched}.

\subsection{Convergence of \texorpdfstring{$\wt M_k$}{}}
\label{subsec:annealed_non-intersecting}
Let $p \in (1/2, 1)$. The sequence $\wt M_k$ from~\eqref{eq:wtM_k} is strongly dependent on the choice of the positive sequence $\wt N_k$. Therefore, finding $\wt N_k$ such that the limiting behaviour of $\wt M_k$ is non-trivial is a central matter.
From~\eqref{eq:exp_Ij}, we find that $\wt M_k$ converges to zero in~$L^1$ if $\wt N_k$ grows slower than 
$(p^2+(1-p)^2)^{-k}$.
This gives us an asymptotic lower bound for the sequences $N_k$ of interest.
Nevertheless, as we will see in Proposition~\ref{prop:conv_0_non-intersecting}, the threshold for this trivial convergence is actually higher.

Given a sequence $a_k>0$, we define for $k\ge1$
\[
\alpha(k)=\sqrt{k}\cdot\log_{p/(1-p)}(a_k),\qquad
\beta(k)\in\R,\qquad
\gamma(k)=(p/(1-p))^{\alpha(k)+\beta(k)}.
\]
Let $\wt G_k(\beta)$ and $\wt H_k(\beta)$ be the events defined as in~\eqref{eq:wtG_k_and_wtH_k}. In some of the upcoming claims, the sequence~$\beta(k)$ will depend on $a_k$, however, Lemmas~\ref{lem:useful_result_1} and~\ref{lem:useful_result_2} work for any~$\beta(k)$.

\begin{lemma}\label{lem:useful_result_1}
Assume there exists $C>0$ such that $\wt N_k\le C 2^{k\cdot H(p)}a_k^{\sqrt{k}}$ for all sufficiently large~$k\ge1$. Then,
$$\P(\wt M_k\neq0|\wt G_k)\le C \gamma(k).$$
\end{lemma}
\begin{proof}
Reasoning similarly to~\eqref{eq:exp_Ij} with $G_k(\beta)$ from~\eqref{eq:G_k_and_H_k}, we get for every $k,j\ge1$
\begin{equation}\label{eq:conditioning}
\begin{split}
\P\bigl(W=Z^{(j)},\wt G_k\bigr)&\textstyle=
\sum_{\omega\in  G_k}
\P\bigl(W=Z^{(j)}\bigl|W=\omega\bigr)\P(W=\omega)\\
&\textstyle=\sum_{i=0}^{\lfloor pk+\beta(k)\rfloor}
\binom{k}{i}\bigl(p^{i}(1-p)^{k-i}\bigr)^2,
\end{split}
\end{equation}
under the convention $\binom{k}{i}=0$ for $i>k$.
We have that
$\P(\wt G_k)=
\sum_{i=0}^{\lfloor pk+\beta(k)\rfloor}
\binom{k}{i}p^{i}(1-p)^{k-i}.
$
Since the map $x\mapsto p^x(1-p)^{k-x}$ is increasing,
\[
\P\bigl(W=Z^{(j)}\bigl|\wt G_k\bigr)=
\frac{\P\bigl(W=Z^{(j)},\wt G_k\bigr)}{\P(\wt G_k)}\le 
p^{pk+\beta(k)}(1-p)^{k-pk-\beta(k)}.
\]
By the hypothesis,
\begin{equation}\label{eq:calculation_N_k}
\wt N_k\le C2^{k\cdot H(p)}a_k^{\sqrt{k}}=C(p^p(1-p)^{1-p})^{-k}(p/(1-p))^{\alpha(k)},    
\end{equation}
and a calculation yields that
\begin{equation*}\label{eq:N_k_times}
\wt{N}_k\cdot \P\bigl(W=Z^{(1)}\bigl|\wt G_k\bigr)\le
C(p/(1-p))^{\alpha(k)+\beta(k)}
=C\gamma(k).
\end{equation*}
Therefore,
\begin{align*}\textstyle
 \E\bigl[\wt M_k\bigl|\wt G_k\bigr]&\textstyle=
\sum_{j=1}^{\wt N_k}\P\bigl(W=Z^{(j)}\bigl|\wt G_k\bigr)
\textstyle=
\wt N_k\cdot \P\bigl(W=Z^{(1)}\bigl|\wt G_k\bigr)
\textstyle\le C\gamma(k).   
\end{align*}
We conclude by means of Markov's inequality:
\[
\P(\wt M_k\neq0|\wt G_k)=\P\bigl(\wt M_k\ge1\bigl|\wt G_k\bigr)\le  \E\bigl[\wt M_k\bigl|\wt G_k\bigr]\le C\gamma(k).\tag*{\qedhere}
\]
\end{proof}

\begin{lemma}\label{lem:useful_result_2}
Assume there exists $C_1>0$ such that $\wt N_k\ge C_1 2^{k\cdot H(p)}a_k^{\sqrt{k}}$ for all~$k\ge1$ sufficiently large.  Then, for any~$n\ge0$ there is $C_2>0$ such that
$$\P\bigl(\wt M_k=n|\wt H_k\bigr)\le C_2 \wt N_k^n \exp\bigl\{-C_1\gamma(k)\}.$$

\end{lemma}

\begin{proof}

For a fixed $n\ge0$ we see, similarly to~\eqref{eq:distribution_wtM}, that 
\begin{align*}\label{eq:Jon_k}\textstyle
\P(\wt M_k=n,\wt H_k)\le\binom{\wt{N}_k}{n}
\sum_{i=\lceil pk+\beta(k)\rceil}^{k}\binom{k}{i}
(p^{i}(1-p)^{k-i})^{n+1}(1-p^{i}(1-p)^{k-i})^{\wt{N}_k-n}.
\end{align*}
If $i$ lies in the range of the sum, then
\begin{equation*}
(1-p^{i}(1-p)^{k-i})^{\wt{N}_k-n}\le 
(1-p^{pk+\beta(k)}(1-p)^{k-pk-\beta(k)})^{\wt{N}_k-n}=d_k.
\end{equation*}
We have that 
$\P(\wt H_k)\le
\sum_{i=\ceil{pk+\beta(k)}}^k
\binom{k}{i}p^{i}(1-p)^{k-i}$. Using that $(p^{i}(1-p)^{k-i})^{n+1}\le p^{i}(1-p)^{k-i}$,
\[
\P\bigl(\wt M_k=n\mid\wt H_k\bigr)=
\frac{\P\bigl(\wt M_k=n,\wt H_k\bigr)}{\P(\wt H_k)}\le 
\binom{\wt N_k}{n}d_k.
\]

By our hypothesis and Equation~\eqref{eq:calculation_N_k}, a calculation yields that
\[
\wt{N}_k\cdot p^{pk+\beta(k)}(1-p)^{k-pk-\beta(k)}\ge
C_1(p/(1-p))^{\alpha(k)+\beta(k)}=
C_1\gamma(k).
\]

Therefore, $d_k=O(\exp\{-C_1\gamma(k)\})$.
By $\binom{\wt N_k}{n}\le \frac{\wt N_k^n}{n!}$, we conclude
\[
\P(\wt M_k=n|\wt H_k)=O( \wt{N}_k^n\cdot 
d_k)
=O\bigl(\wt N_k^n \exp\bigl\{-C_1\gamma(k)\}\bigr),
\]
as desired.
\end{proof}

\begin{remark}\label{rem:application_intersecting}
Lemma~\ref{lem:useful_result_1} remains valid if we replace $\wt M_k$ with~$M_k$. This follows from the key identity $\P(W = Z^{(j)} \mid \wt G_k) = \E_k[I_j \mid  G_k]$, where~$I_j$ is the indicator defined in~\eqref{eq:Indicator}, and $G_k$ is from~\eqref{eq:G_k_and_H_k}. With this, the proof carries through in the same way. On the other hand, Lemma~\ref{lem:useful_result_2} cannot be similarly adapted for $M_k$ and $H_k$, as its proof relies on the explicit distribution of the random variable~$\wt M_k$.
\end{remark}

\begin{proposition}\label{prop:conv_0_non-intersecting} If $\wt{N}_k=o\bigl(2^{k\cdot H(p)}a^{\sqrt{k}}\bigr)$ for all $a>0$, then $\P(\wt M_k=0)\to1$.
\end{proposition}

\begin{proof}
Let $a_k\to 0$ be a positive sequence such that $\wt N_k=\lfloor 2^{k\cdot H(p)}a_k^{\sqrt{k}}\rfloor $. Define for $k\ge1$ 
\begin{equation}\label{eq:abc}
\alpha(k)=\sqrt{k}\cdot\log_{p/(1-p)}(a_k),\qquad\beta(k)=-\alpha(k)/2,\qquad
\gamma(k)=(p/(1-p))^{\alpha(k)/2}.    
\end{equation}
We define $\wt G_k(\beta)$ as in~\eqref{eq:wtG_k_and_wtH_k}. Note that  $\gamma(k)\to0$ and $\P(\wt G_k)\to1$ by the CLT.
Applying Lemma~\ref{lem:useful_result_1}, 
\[
\P(\wt M_k\neq0|\wt G_k)\le 2\gamma(k)\to 0.
\]
It follows that $\P(\wt M_k\neq0)\to 0$, which proves our statement.
\end{proof}

\begin{proposition}\label{prop:vague_0_non-intersecting}
If $\wt{N}_k=\omega\bigl(2^{k\cdot H(p)}a^{\sqrt{k}}\bigr)$ for all $a>0$,  then $\P(\wt M_k\ge n)\to 1$ for any $n\ge0$.
\end{proposition}
\begin{proof}
Let $a_k\to \infty$ be a positive sequence such that $\wt N_k=\lceil 2^{k\cdot H(p)}a_k^{\sqrt{k}}\rceil$. Define for $k\ge1$
\[
\alpha(k)=\sqrt{k}\cdot\log_{p/(1-p)}(a_k),\qquad\beta(k)=-\alpha(k)/2,\qquad \gamma(k)=(p/(1-p))^{\alpha(k)/2}.
\]
Let us fix $n\ge0$ and let $\wt H_k(\beta)$ be from~\eqref{eq:wtG_k_and_wtH_k}. Note that $\gamma(k)\to\infty$ and $\P(\wt H_k)\to1$ by the CLT.
We additionally assume that $\wt N_k\le \exp\{\gamma(k)/(n+1)\}$.
Applying Lemma~\ref{lem:useful_result_2}, for any $j\le n$ there is $C_j>0$ such that for any $j\le n$
$$P_k\bigl(\wt M_k=j|\wt H_k\bigr)\le C_j \wt N_k^j\exp\{-\gamma(k)\}\to 0.$$
It follows that $\P(\wt M_k= j)\to 0$.
Therefore,  $P(M_k<n)\to 0$ as well, and the statement follows by
\[\textstyle
\P(\wt M_k\ge n)=1-P(\wt M_k<n)\to1.
\]

Let us now drop the additional assumption on $\wt N_k$. Consider a new sequence \[
\wt N_k'=\min\prs{\wt N_k,\exp\crb{\gamma(k)/(n+1)}},
\]
and define from it a new 
\[\textstyle
 \wt{M}_k'=\#\bigl\{1\le j\le \wt N_k':W=Z^{(j)}\bigr\}=
    \sum_{j=1}^{\wt N_k'}\1_{\{W=Z^{(j)}\}},
\] similarly to~\eqref{eq:wtM_k}. Here $Z^{(j)}$, $j\ge1$, are the iid random vectors from Subsection~\ref{subsec:non_intersecting}. By what seen above, $\P(\wt M_k'\ge n)\to 1$. Since $\wt M_k'\le \wt M_k$, it follows that 
\[
\P(\wt M_k\ge n)\ge\P(\wt M_k'\ge n)\to 1,
\]
finishing the proof.
\end{proof}

\begin{proposition}\label{prop:annealed_convergence_non_intersecting}
Let $a>0$. If $ \wt N_k=\Theta( 2^{k\cdot H(p)}a^{\sqrt{k}})$ , then:
\begin{itemize}
    \item\label{item:partial_escape_1}$\lim_k\P(\wt{M}_k=0)=\Phi\bigl(-(\log_{p/(1-p)}(a))(p(1-p))^{-1/2}\bigr)$, and
    \item\label{item:partial_escape_2}$\lim_k\P(\wt M_k= n)=0$ for any $n\ge1$.
\end{itemize}
\end{proposition}

\begin{proof}
By our assumption, there are constants $a,C_1,C_2>0$ such that 
$$C_2\le \wt N_k/(2^{k\cdot H(p)}a^{\sqrt{k}})\le C_1,$$ for all $k$ large enough. 
For $c=\log_{p/(1-p)}(a)$, we define the sequences 
\begin{equation}\label{eq:abc_bis}
\alpha(k)=c\sqrt{k},\qquad\beta(k)=-c\sqrt{k}-k^{1/4},
\qquad\gamma(k)=(p/(1-p))^{-k^{1/4}},
\end{equation}
so $\gamma(k)\to0$. Let $\wt G_k(\beta)$ be the event defined in~\eqref{eq:wtG_k_and_wtH_k}.
By applying Lemma~\ref{lem:useful_result_1}, it follows that
$
\P(\wt M_k\neq0|\wt G_k)\le C_1 \gamma(k)\to0
$, which gives
\begin{equation}\label{eq:intersecting_conv_0}
\P(\wt M_k=0|\wt G_k)\to1.
\end{equation}

On the other hand, let 
\[
\widehat\alpha(k)=\alpha(k)\qquad
\widehat\beta(k)=-c\sqrt{k}+k^{1/4},\qquad
\widehat\gamma(k)=(p/(1-p))^{k^{1/4}},
\]
so $\widehat\gamma(k)\to \infty$. Let $\wt H_k(\widehat\beta)$ be from~\eqref{eq:wtG_k_and_wtH_k}. For a fixed $n\ge0$, Lemma~\ref{lem:useful_result_2} gives that there is $C_3>0$ such that 
\begin{equation}\label{eq:partial_escape_2}\P(\wt M_k=n|\wt H_k)\le C_3\wt N_k^n \exp\bigl\{-C_2\gamma(k)\bigr\}\to0.
\end{equation}

Let $\Phi_p$ be the cumulative distribution function defined in~\eqref{eq:cumulant_p}. Using Remark~\ref{rem:application_CLT}, we get 
$\P(\wt G_k(\beta))\to\Phi_p(-c)$ and $\P(\wt H_k(\widehat\beta))\to 1-\Phi_p(-c)$. By~\eqref{eq:partial_escape_2},
\[\begin{split}
\P(\wt M_k=n)&=\P(\wt M_k=n|\wt G_k(\beta))\P(\wt G_k(\beta))+\P(\wt M_k=n|\wt H_k(\widehat\beta))\P(\wt H_k(\widehat\beta))+o(1)\\
&=\P(\wt M_k=n|\wt G_k(\beta))\P(\wt G_k(\beta))+o(1).
\end{split}
\]
The first point is proven by setting $n=0$ in the above identity and using~\eqref{eq:intersecting_conv_0}. The second point follows in the same way letting $n\ge1$.
\end{proof}

\begin{remark}
In the proof of Proposition~\ref{prop:annealed_convergence_non_intersecting} we define both sequences $\beta$ and $\widehat\beta$ using the summand~$k^{1/4}$. This can be replaced by any $b_k=o(\sqrt{k})$ such that the convergence in~\eqref{eq:partial_escape_2} is satisfied.
\end{remark}

\subsection{Convergence of \texorpdfstring{$M_k$}{}}
\label{subsec:annealed_intersecting}
Let $p \in (1/2, 1)$, and consider the sequence $M_k$ from~\eqref{eq:Mk}, defined in terms of the sequence $\{N_k\}_{k\ge1}$.
We let $\wt N_k=\lfloor N_k/k\rfloor$ and define (on a separate space with probability measure $\P$) a sequence of random variables $\{\wt M_k\}_{k\ge1}$, as in~\eqref{eq:wtM_k}. Using the indicators from~\eqref{eq:Indicator}, we define for any $k\ge1$
\begin{equation}\label{eq:Yk}\textstyle
Y_k=\sum_{j=1}^{\lfloor N_k/k\rfloor}I_{j\cdot k}.    
\end{equation}
It is clear that $Y_k\le M_k$ since it is a sum of a subset of the same indicators. 
Moreover, $Y_k$ is a sum of indicators whose strings do not overlap, and hence are generated independently from one another under $ P_k$. Therefore $Y_k$ and $\wt M_k$ have the same distribution. In the following two proofs, we show that if $N_k$ grows too slow or too fast, then the distribution of $M_k$ has a trivial behaviour at the limit.

\begin{proof}[Proof of Proposition~\ref{prop:convergence_annealed}-\ref{item:conv_0_annealed}]
 We reason similarly to Proposition~\ref{prop:conv_0_non-intersecting}, outlining the main steps.
Denote 
$ N_k=\lfloor 2^{k\cdot H(p)}a_k^{\sqrt{k}}\rfloor $ for some positive sequence $a_k\to 0$. Let  $\alpha,\beta,\gamma$ be as in~\eqref{eq:abc},
and define $ G_k(\beta)$ from~\eqref{eq:G_k_and_H_k}. Note that $\gamma(k)\to0$ and $\P( G_k)\to1$ by the CLT.
By Remark~\ref{rem:application_intersecting}, we can apply  Lemma~\ref{lem:useful_result_1} to $M_k$, getting that
\[
\P(M_k\neq0| G_k)\le \gamma(k)\to 0.
\]
Point~\ref{item:conv_0_annealed} follows.
\end{proof}

\begin{proof}[Proof of Proposition~\ref{prop:convergence_annealed}-\ref{item:vague_0_annealed}]
Let us  assume $N_k=\omega(2^{k\cdot H(p)}a^{\sqrt{k}})$ and consider the sequence $Y_k$ as in~\eqref{eq:Yk}.
For $\wt N_k=\lfloor N_k/k\rfloor$, we let $\wt M_k$ be as in~\eqref{eq:wtM_k}.
Since $k=o(a^{\sqrt{k}})$ for any $a>1$, it follows that {$\wt{N}_k=\omega\bigl(2^{k\cdot H(p)}a^{\sqrt{k}}\bigr)$ for all $a>0$ as well}. 
Hence, Proposition~\ref{prop:vague_0_non-intersecting} yields that 
$ P_k(\wt M_k\ge  n)\to1$ for any $n\ge0$. Since $Y_k\le M_k$,  we conclude the proof of~\ref{item:vague_0_annealed} by
\[
 P_k(M_k\ge n)\ge P_k(Y_k\ge n)= \P(\wt M_k\ge  n)\to1.\tag*{\qedhere}
\]
\end{proof}

Finally, we show that if $N_k$ is chosen in a suitable asymptotic class, then the sequence~$M_k$ displays a limiting atom at zero and a partial escape of mass to infinity.

\begin{proof}[Proof of Proposition~\ref{prop:convergence_annealed}-\ref{item:partial_escape_annealed}]
By assumption, there are constants $a,C_1,C_2>0$ such that 
$$
C_2\le \ N_k/(2^{k\cdot H(p)}a^{\sqrt{k}})\le C_1,
$$
for all sufficiently large $k$. 
Let $c=\log_{p/(1-p)}(a)$, and define the sequences $\alpha,\beta,\gamma$ as in~\eqref{eq:abc_bis}, so that $\gamma(k)\to0$.
Let $ G_k(\beta)$ be the set defined in~\eqref{eq:G_k_and_H_k}.
By Remark~\ref{rem:application_intersecting}, we can apply  Lemma~\ref{lem:useful_result_1} to $M_k$, getting that
$
P_k( M_k\neq0| G_k)\le C_1 \gamma(k)\to0.  
$
Therefore, 
\begin{equation}\label{eq:intersecting_partial_escape_G_k}
P_k( M_k=0| G_k)\to1.
\end{equation}

Let $n\ge0$, and fix $\delta>0$ such that $a-\delta>0$.
Define $\wt N_k=\lfloor N_k/k\rfloor$, and let $\wt M_k$ be from~\eqref{eq:wtG_k_and_wtH_k}.
Note that $\wt N_k\ge C_2 2^{k\cdot H(p)}(a-\delta)^{\sqrt{k}}$ for all sufficiently large $k$. Let $c_\delta=\log_{p/(1-p)}(a-\delta)<c$ and  define the three sequences
\[
\alpha_\delta(k)=c_\delta\sqrt{k}\qquad
\beta_\delta(k)=-c_\delta\sqrt{k}+k^{1/4},\qquad
\gamma_\delta(k)=(p/(1-p))^{k^{1/4}},
\]
so $\gamma_\delta(k)\to\infty$.
Let $\wt H_k(\beta_\delta)$ and $H_k(\beta_\delta)$ be the events  from~\eqref{eq:wtG_k_and_wtH_k} and~\eqref{eq:G_k_and_H_k}, respectively. 
 By Lemma~\ref{lem:useful_result_2}, there exists $C_3>0$ such that 
\[
 \P(\wt M_k=n|\wt H_k(\beta_\delta))\le C_3\wt N_k^n\exp\{-C_2\gamma_\delta(k)\} \to 0.
\]
It follows that $\P(\wt M_k\ge n|\wt H_k(\beta_\delta))\to 1$.

A calculation--similar to the one in Lemma~\ref{lem:useful_result_2}--and identity~\eqref{eq:identities} show that for any $m\ge0$
\begin{align*}\textstyle
\P(\wt M_k=m\mid\wt H_k(\beta_\delta))=P_k(Y_k=m\mid H_k(\beta_\delta)).
\end{align*}
Therefore,
\[
P_k( Y_k\ge n| H_k(\beta_\delta))= \P(\wt M_k\ge n|\wt H_k(\beta_\delta)).
\]
Using the fact that $M_k\ge Y_k$ for $Y_k$ from~\eqref{eq:Yk}, we obtain
\begin{equation}\label{eq:intersecting_partial_escape_wtH_k}
 P_k( M_k\ge n| H_k(\beta_\delta))\to 1.
\end{equation}

Now define the set
\[
E_k^\delta=\{0,1\}^k\setminus(G_k(\beta)\cup H_k(\beta_\delta))
=\crb{\omega:\ -c\sqrt{k}-k^{1/4}<|\omega|-pk\le -c_\delta\sqrt{k}+k^{1/4}},
\] 
For $\Phi_p$ the cumulative distribution function defined in~\eqref{eq:cumulant_p}, Remark~\ref{remark:application_CLT} yields
$$P_k(E_k^\delta)=1-\bigl(P_k(G_k(\beta))+\P(H_k(\beta_\delta))\bigr)\to \Phi_p(-c_\delta)-\Phi_p(-c).$$
Fix $\epsilon>0$ and choose $\delta>0$ such that $P_k(E^\delta_k)\le\epsilon$ for all sufficiently large $k$.  
This $\delta$ exists by the fact that $\Phi_p$ is continuous. 
By~\eqref{eq:intersecting_partial_escape_wtH_k}, we get for any $n\ge0$ 
\begin{equation}\label{eq:total_prob}
\begin{split}
P_k(M_k=n)&=P_k\bigl(M_k=n\mid G_k\bigr) P_k(G_k)+
P_k\bigl(M_k=n\mid E^\delta_k\bigr)P_k(E^\delta_k)+o(1)\\
&\le P_k\bigl(M_k=n\mid G_k\bigr) P_k(G_k)+\epsilon+o(1),
\end{split}
\end{equation}
for all sufficiently large $k$. By~\eqref{eq:intersecting_partial_escape_G_k} and Remark~\ref{remark:application_CLT}, we note that 
$$P_k\bigl(M_k=0\mid G_k\bigr) P_k(G_k)\to\Phi_p(-c).$$ By~\eqref{eq:total_prob} we get $P_k(M_k=0)\to\Phi_p(-c)$, that is the first part of~\ref{item:partial_escape_1_intersec}.
Now let $n\ge1$. Applying~\eqref{eq:intersecting_partial_escape_G_k} to~\eqref{eq:total_prob}, we conclude that $P_k(M_k=n)\to 0$, which finishes the proof.
\end{proof}

\section{Conditional Poisson convergence}\label{sec:cond_Poi_conv}
Let $\wt M_k$ and $M_k$ be the random variables defined in \eqref{eq:wtM_k} and \eqref{eq:Mk} respectively.
As we see in Corollary~\ref{cor:non_Poisson}, for $p\neq1/2$ and any sequence $N_k$, $M_k$ does not converge to the  
Poisson distribution, as happens in the case $p=1/2$ (with $N_k=2^k$).
From the proofs in Section~\ref{sec:annealed} it is also clear why: as we saw in \eqref{eq:W=Z}, the expected number of appearances of a word depends exponentially on the word's Hamming weight (see Definition~\ref{def:HammingWeight}), and so for any choice of $N_k$, most words will appear either too often or too rarely, leading to the result of Proposition~\ref{prop:convergence_annealed}. 

Next, we prove that when this factor is controlled, i.e.\@~when we consider the subset with Hamming weight fixed (depending only on $k$), the limiting distribution \textbf{is} Poisson.

\begin{claim}\label{claim:cond_non_intersecting}
    Let $p\in (0,1)$, $\lambda>0$ and let $m_k$ be a rising sequence in $\N$.
    Define 
    $$A_k=\crb{\omega\in \crb{0,1}^k :\ |\omega|=m_k },\quad\wt N_k=\floor{\frac{\lambda}{p^{m_k}(1-p)^{k-m_k}}}.$$ 
    Then, for all $n\in \N\cup \crb{0}$, $ \P(\wt M_k=n\lvert A_k)\to e^{-\lambda}\frac{\lambda^n}{n!}$.
\end{claim}
\noindent This is apparent--$\wt M_k$ is a sum of i.i.d indicators, and by~\eqref{eq:W=Z}
\[
\E[\wt M_k\mid W\in A_k]=\wt N_k\cdot \P(W=Z^{(1)}\mid W\in A_k)
=\wt N_k\cdot p^{m_k}(1-p)^{k-m_k}\to \lambda.
\]
The claim follows by the Poisson limit theorem.

Recall from \eqref{eq:F_k} that for $c\in \R$ and  $n_k=\lfloor pk-c\sqrt{k}\rfloor$, 
$$F_k:=\bigl\{\omega\in\{0,1\}^k:|\omega|=n_k\bigr\}.$$
The local De Moivre-Laplace formula~\cite[Theorem 3.1.2]{Dur19}\footnote{\cite[Theorem 3.1.2]{Dur19} states a version of the local De Moivre-Laplace formula for a random variable $Y$ with Rademacher distribution: $\P(Y=1)=p$, $\P(Y=-1)=1-p$. Formula~\eqref{eq:local_DM-L} can be derived from it via the standard relation $Y\overset{d}{=}2X-1$, where $X\sim\Ber(p)$.}
yields that, for any $n\ge0$,
\begin{equation}\label{eq:local_DM-L}
\Ber^k(\omega:|\omega|=n) \sim (2\pi k p(1-p))^{-1/2} \exp\bigl\{-(n-pk)^2/(2 k p(1-p))\bigr\},
\end{equation}
as $k\to\infty$.   It follows that 
\begin{equation}\label{eq:local_CLT}
\Ber^k(F_k)=\Theta(1/\sqrt{k}).
\end{equation}

Recall the probability measure $\ddot P_k=P_k(\cdot|F_k)$, and the sequence $\ddot M_k$ as defined in Subsection~\ref{subsec:intersecting}.
We denote by $\ddot\E_k$ the expectation according to $\ddot P_k$. For $i\ge1$ and $I_i$ from~\eqref{eq:Indicator}, we denote 
\begin{equation}\label{def:q}
    q=\ddot \E_k[I_i]=p^{n_k}(1-p)^{k-n_k},
\end{equation}
which is independent of $i$.

The aim of this section is to prove an analogous result  to Claim~\ref{claim:cond_non_intersecting} in the intersecting case-- the annealed version of Theorem~\ref{thm:quenched_conditional_poisson}.
\begin{proposition}\label{prop:conditional_poisson_annealed}
Let $p\in (0,1)$, $\lambda>0$, and let 
$N_k=\floor{\lambda(p^{n_k}(1-p)^{k-n_k})}=\floor{\lambda/q}$. 
Then $\ddot M_k\xrightarrow{d} \Po(\lambda)$.
\end{proposition}
In this case the Poisson limit theorem is not applicable, as the indicators in the definition of~$\ddot M_k$ are not all independent. As such, the proof of Proposition~\ref{prop:conditional_poisson_annealed} is more involved than that of Claim~\ref{claim:cond_non_intersecting}, and requires some preparation.
 Without loss of generality, in the following we deal with the case $p>1/2$ (see Remark~\ref{rem:p=1/2} for $p=1/2$).

\begin{definition}
Let $\{ I_j \}_{j \in J}$ be a family of random variables on the same probability space. A \textit{dependency graph} for such a family is a graph $L$ with underlying vertex set $J$, such that for any pair of
disjoint subsets $A, B \subseteq J$ of vertices with no edges $(a, b)$, $a \in A$, $b \in B$ connecting them,
the subfamilies $\{ I_i \}_{i \in A}$ and $\{ I_j \}_{j \in B}$ are mutually independent.
\end{definition}
 \noindent Note that a dependency graph is in general not unique.

We denote with $d_{TV}$ for the \textit{total variation} distance on the space of probability measures of a measurable space $(\Lambda,\mathcal F)$. That is
\begin{equation*}
d_{TV}(P,Q)=\sup_{ F\in\mathcal F}\left|P(F)-Q(F)\right|.
\end{equation*}
If $X$ and $Y$ are real random variables (possibly defined on different spaces), we write $d_{TV}(X,Y)$ to indicate the total variation distance between the laws of~$X$ and~$Y$, that is

\begin{equation*}
d_{TV}(X,Y)=\sup_{ A}\left|\P(X\in A)-\P(Y\in A)\right|,
\end{equation*}
where $A$ runs over measurable subsets of $\R$.
Proving that $\lim_{k\to\infty}d_{TV}(X_k,Y)=0$ for a sequence~$X_k$ of random variables clearly implies that $X_k$ converges in distribution to $Y$.

To prove Proposition~\ref{prop:conditional_poisson_annealed}, we  will apply the following general result for Poisson convergence, as done in \cite{AlvBecMer23}.

\begin{theorem}[{\cite[Theorem 6.23]{JanLucRuc00}}]\label{thm:dtv}
Let $\Po(\lambda)$ be a Poisson random variable with mean $\lambda>0$.
Let  $\crb{ I_j}_{j \in J}$ be a family of indicator random variables on a given  probability space and
let $L$ be a  dependency graph of $\crb{ I_j}_{j \in J}$, with underlying vertex set $J$.
Suppose that the random variable $X_J = \sum_{j\in J} I_j$
satisfies $\lambda = \E \sqb{ X_J} = \sum_{j\in J} \E \sqb{ I_j } $.
Then,
\begin{align*}
    d_{TV} (X_J, \Po(\lambda) )
    \leq \min\crb{ 1, \lambda^{-1}}\prs{  \sum_{j \in J} \E \sqb{I_j}^2 
    +  \sum_{\substack{ i,j : (i,j) \in \text{edges}(L)  }}
\Big(     \E \sqb{I_i I_j} + \E \sqb{I_i}  \E \sqb{ I_j } 
    \Big)}
\end{align*}
\end{theorem}

\begin{definition}\label{def:Z}
For $1\le \ell\le k$, define $Z_\ell^k=F_k\cap\crb{\omega:\prs{\omega_1,...,\omega_\ell}=\prs{\omega_{k-\ell+1},...,\omega_k}}$, i.e.\@~the set of words with the first $\ell$ characters are identical to the last $\ell$.
These are the only words for which $I_i(x,\omega)I_j(x,\omega)$ is not identically zero when $k-(j-i)=\ell$ (see illustration).
\end{definition}
\begin{center}

\begin{tikzpicture}
    \draw[->] (0,0) -- (6,0);
    
    \foreach \x/\label in {1/i, 2/j, 4/i+k, 5/j+k} {
        \node at (\x,0.3) {\label};
    \draw[shift={(\x,0)}] (0,-0.1) -- (0,0.1);}
    \draw[decorate,decoration={brace,mirror,amplitude=10pt}] (2.1,-0.1) -- (3.9,-0.1) node[midway,yshift=-8pt,below] {$\ell$};
\end{tikzpicture}
\end{center}

\begin{lemma}\label{lemma:expI_iI_j}
    For $1\le i<j$, denote by $\ell=k-(j-i)$ and recall from \eqref{def:q} that~$q=\ddot \E_k[I_i]$. For all sufficiently large $k$:
    \begin{enumerate}
        \item If $\ell<\sqrt{k}$, then $\ddot \E_k[I_i I_j]=O\prs{q^{1.9}}$.
        \item\label{item:2.} If $\ell\ge \sqrt{k}$, then $\ddot \E_k[I_i I_j]=O \prs{\frac{q}{k}(p^2+(1-p)^2)^{k^{0.4}}}$.
    \end{enumerate}
\end{lemma}
\begin{proof}
\begin{enumerate}
\item Write $I_i$ for the random variables $I_i(x,\omega)$ from \eqref{eq:Indicator}.
$$\ddot\E_k[I_i\cdot I_j]=\ddot{P}_k(I_i\cdot I_j=1)= \ddot{P}_k(I_i\cdot I_j=1\cap Z_{\ell} ^k)=  
\sum_{\omega\in Z_{\ell} ^k}\ddot{P}_k(I_i\cdot I_j=1|\omega)\ddot{P}_k(\omega).$$
Denote by $u_\ell(\omega)=|\prs{\omega_1,\dots,\omega_\ell}|$. We see that for any $\omega\in Z_\ell ^k$
$$\ddot{P}_k(I_i\cdot I_j=1|\omega)=p^{2n_k-u_{\ell}(\omega)}(1-p)^{2k-2n_k-(\ell-u_\ell(\omega))}\le p^{2n_k}(1-p)^{2k-2n_k-\ell},$$
as $p>1-p$. Therefore
$$\sum_{\omega\in Z_{\ell} ^k}\ddot{P}_k(I_i\cdot I_j=1|\omega)\ddot{P}_k(\omega)\le \sum_{\omega\in Z_{\ell} ^k}p^{2n_k}(1-p)^{2k-2n_k-\ell}\ddot{P}_k(\omega).$$
We have $Z_{\ell} ^k\subset F_k$, and so $\ddot{P}_k(\omega)|Z_{\ell} ^k|\le 1$. Since the distribution on words in $F_k$ is uniform (as $\ddot{P}_k(\omega)\equiv~q$), we conclude 
\begin{align*}
\ddot \E_k[I_i I_j]&\le \sum_{\omega\in Z_{\ell} ^k}p^{2n_k}(1-p)^{2k-2n_k-\ell}\ddot{P}_k(\omega)=|Z^k_\ell|\cdot q\cdot p^{2n_k}(1-p)^{2k-2n_k-\ell}\\
&\le p^{2n_k}(1-p)^{2k-2n_k-\ell}=q^2(1-p)^{-\ell}=O(q^{1.9}),
\end{align*}
as $p>1-p$ and $\ell=o(k)$ by assumption.
\item For any $S\subset F_k$, $\ddot{P}_k(I_i=1|S)=\ddot{P}_k(I_i=1)$. In particular $I_i$ is independent of $Z_{\ell} ^k$. Therefore
$$\ddot \E_k[I_i I_j]=\ddot{P}_k(I_i\cdot I_j=1)= \ddot{P}_k(I_i\cdot I_j=1\cap Z_{\ell} ^k)\le \ddot{P}_k(I_i=1\cap Z_{\ell} ^k)=\ddot{P}_k(I_i=1)\ddot{P}_k( Z_{\ell} ^k).$$
Remembering the definition of $\ddot P_k$ and using~\eqref{eq:local_CLT}, we see that there is $C>0$ such that for all sufficiently large $k\ge1$,
$$\ddot{P}_k(I_i=1)\ddot{P}_k( Z_{\ell} ^k)=q\frac{P_k( Z_{\ell} ^k)}{P_k(F_k)}\le qC\sqrt{k}(p^2+(1-p)^2)^{\ell},$$ 
where we use that $P_k( Z_{\ell} ^k)\le (p^2+(1-p)^2)^\ell$ by digit agreement (similarly to \eqref{eq:exp_Ij}).
Since $\ell\ge \sqrt{k}$, we see that
$$\ddot \E_k[I_i I_j]\le qC\sqrt{k}(p^2+(1-p)^2)^{\ell}=o\prs{\frac{q}{k}(p^2+(1-p)^2)^{k^{0.4}}},$$
granting the result.
\qedhere
    \end{enumerate}
\end{proof}

\begin{proof}[Proof of Proposition~\ref{prop:conditional_poisson_annealed}]
    The indicators $I_i,I_j$ are independent w.r.t.\ $\ddot{P}_k$ unless $|i-j|< k$.
    We may therefore use Theorem~\ref{thm:dtv} for the indicators $\crb{I_i:1\le i\le N_k}$, with the dependency graph having the edge set $\crb{(i,j):|j-i|<k}$. 
    Denote $\lambda_k=\ddot \E_k[M_k]$. We notice that 
    $$\lambda_k=q\cdot\floor{\lambda/q}\in \sqb{\prs{\lambda/q-1}q,\lambda},$$
    and so $\lambda_k\to \lambda$. Plugging in the bounds from Lemma~\ref{lemma:expI_iI_j} into Theorem~\ref{thm:dtv} we obtain
    \begin{align*}
    d_{TV} (\ddot M_k, \Po(\lambda_k) )
    \leq& \min\crb{ 1, \lambda_k^{-1}}\prs{  \sum_{m =1} ^{N_k} q^2 
    +  \sum_{\substack{ i,j: |i-j|<k}}
\Big( q^2+\ddot \E_k \sqb{I_i I_j} 
    \Big)}\\
\le& \min\crb{ 1, \lambda_k^{-1}}\big( N_k q^2 + kN_k q^2\big)\\ 
    &+N_k\Big(c_1\sqrt{k}\cdot q^{1.9} +  c_2(k-\sqrt{k})
    \frac{q}{k} \big(p^2+(1-p)^2\big)\Big)^{k^{0.4}},
\end{align*}
where $c_1,c_2>0$ exist by Lemma~\ref{lemma:expI_iI_j}. Since $q=\Theta\prs{N_k^{-1}}$ and $N_k$ is exponential, it follows that $q^2N_k(1+k)$ and $N_k\cdot c_1\sqrt{k}\cdot q^{1.9}$ go to zero.
Note that
\[
N_kc_2(k-\sqrt{k})\frac{q}{k} (p^2+(1-p)^2)^{k^{0.4}}=\Theta\prs{\frac{(k-\sqrt{k})}{k} (p^2+(1-p)^2)^{k^{0.4}}}\to~0.
\]
So, $d_{TV} (\ddot M_k, \Po(\lambda_k) )\to 0$.

By \cite[formula (5)]{Roos03}, $\lambda_k\to \lambda$ entails $d_{TV}(\Po(\lambda_k),\Po(\lambda))\to~0$.
 The total variation distance is a metric, and in particular satisfies the triangle inequality. So, 
$$d_{TV}(\ddot M_k,\Po(\lambda))\le d_{TV}(\ddot M_k,\Po(\lambda_k))+d_{TV}(\Po(\lambda_k),\Po(\lambda))\to 0,$$
finishing the proof.
\end{proof}

\begin{remark}\label{rem:p=1/2}
When $p=1/2$, we can move directly to the proof of Proposition~\ref{prop:conditional_poisson_annealed}, without relying on Lemma~\ref{lemma:expI_iI_j}. In fact, in this case $N_k=\floor{\lambda\cdot 2^{-k}}$ and for any $\omega\in Z_\ell^k$
\[
P_k(I_i\cdot I_j|\omega)=2^{-2k+\ell}
\qquad\text{and}\qquad
|Z_\ell^k|=2^{k-\ell}
.\]
By $\ddot P_k(\omega)=2^{-k}$, it follows that
\[
\ddot\E[I_i\cdot I_j]=
\sum_{\omega\in Z_\ell^k}\ddot P_k(I_i\cdot I_j|\omega)\ddot P_k(\omega)=
2^{-2k},
\]
similarly to~\cite{AlvBecMer23}.
This simplifies the calculation when applying Theorem~\ref{thm:dtv}. 
\end{remark}

\section{Quenched results}\label{sec:quenched}
In this section we pass the results of Sections~\ref{sec:annealed} and~\ref{sec:cond_Poi_conv} to their corresponding quenched versions, thus proving all our main theorems. 
We follow the ideas from~\cite{AlvBecMer23}, utilising the Borel-Cantelli lemma~\cite[Chapter 3, Lemma 1] {Fel71} and the classical concentration inequality of McDiarmid~\cite{McD89}. For a sequence of events $\{E_k\}_{k\ge1}$ such that $\sum_{k=1}^\infty\P(E_k)<\infty$, the Borel-Cantelli lemma states that 
$\P(\limsup_k E_k)=0$, that is the probability that infinitely many events occur is zero.

\begin{proposition}[McDiarmid's inequality]\label{prop:McD}
For $m\ge1$, let $X_1,\dots, X_m$ be independent random variables taking values in a set $\Omega$. Let $f\colon\Omega^m\to\R$ be a function and suppose that there is $c>0$ such that  
\begin{equation}\label{eq:inequality_single_coordinate}
|f(x)-f(x')|<c,    
\end{equation}
for any $x,x'\in\Omega^m$, which differ only in a single coordinate. Write $X=(X_1,\dots, X_m)$, and let~$\P$ be the underlying probability measure. Then, for any $t\ge0$ 
\[
\P\bigl(|f(X)-\E[f(X)]|> t\bigr)\le 2\exp\bigl\{-2t^2/(mc^2)\bigr\}.
\]
\end{proposition}

Given a sequence $\{N_k\}_{k\ge1}$ of positive numbers, we consider the sequence $M_k$ as defined in~\eqref{eq:Mk}. Let $M_k^x$ defined as in~\eqref{eq:M_k^x}, so  $M_k^x(\omega)=M_k(x,\omega)$ for all $(x,\omega)\in\Omega_k$. 

For a set $\Lambda$ and  $A\subset \Lambda$, we let $\1_A\colon\Lambda\to\{0,1\}$ be the indicator function
\[
\1_A(x)=\begin{cases}
1 &\text{if }x\in A,\\
0 &\text{if }x\in\Lambda\setminus A.
\end{cases}
\]

\begin{lemma}\label{lem:Borel_Cantelli_applied}
Let $p\in(0,1)$. For $k\ge1$ define  $\beta(k)=\sqrt{k}\log k$ and the set $G_k(\beta)$ as in~\eqref{eq:G_k_and_H_k}. For a fixed $n\ge0$, we define 
$f_k\colon\{0,1\}^{\N}\to[0,1]$ as
\begin{equation}\label{eq:f_k}
f_{k}(x)=\Ber^k\bigl((\omega:M_k^x=n)\cap\ G_k\bigr).    
\end{equation} 
If
\begin{equation}\label{eq:Borel_Cantelli}
\sum_{k=1}^\infty
\Ber^{\N}\Bigl(x:\bigl|f_k(x)-
{\textstyle\int f_k\diff\Ber^\N}\bigr|>1/k\Bigr)<\infty,
\end{equation}
then for $\Ber(p)^\N$-a.e.\@~$x\in\{0,1\}^\N$ and as $k\to\infty$,
\begin{equation*}\label{eq:Mk_as_Mkx}
|\Ber^k(M^x_k=n)-P_k(M_k=n)|\to 0.    
\end{equation*}
\end{lemma}

\begin{proof}
This result follows by an application of the Borel-Cantelli lemma. We explain the main steps.
Let $\widehat G_k=\{0,1\}^\N\times G_k$, so the CLT gives that $\Ber^k(G_k)=P_k(\widehat G_k)\to 1$.
Note that to prove the thesis it suffices to show that for $\Ber(p)^\N$-a.e.\@~$x$,
\begin{equation}\label{eq:Mk_as_Mkx_conditioned}
|f_k(x)-P_k(M_k=n, \widehat G_k)|\to0.   
\end{equation}
This follows by the triangular inequality and
\[
|\Ber^k(M_k^x=n)-f_k(x)|\to 0,
\qquad\qquad
|P_k(M_k=n, \widehat G_k)-P_k(M_k=n)|\to0.
\]

By assumption~\ref{eq:Borel_Cantelli},
the Borel-Cantelli lemma yields that 
$|f_k(x)-\int f_k\diff\Ber^\N|\to 0$ for a.e.\@~$x$, which  is exactly~\eqref{eq:Mk_as_Mkx_conditioned}.
Since for any $(x,\omega)\in \Omega_k$
$$\1_{\{M_k^x=n,\ G_k\}}(\omega)=
\1_{\{M_k=n,\ \widehat G_k\}}(x,\omega),$$
applying Tonelli's theorem we get
\begin{equation}\label{eq:Tonelli}
\begin{split}
\int_{\{0,1\}^\N} f_k\diff\Ber^\N&=
\int_{\{0,1\}^\N}\int_{\{0,1\}^k}
\1_{\{M_k^x=n,\ G_k\}}(\omega)\diff\Ber^k(\omega)\diff\Ber^\N(x)\\
&=\int_{\Omega_k}\1_{\{M_k=n,\ \widehat G_k\}}(x,\omega)\diff P_k(x,\omega)=P_k(M_k=n,\ \widehat G_k).   
\end{split}
\end{equation}
This implies~\eqref{eq:Mk_as_Mkx_conditioned}, thus finishing the proof.
\end{proof}

Reasoning as in~\cite{AlvBecMer23}[Proof of Theorem 1], in the following we utilise the McDiarmid's inequality to prove condition~\eqref{eq:Borel_Cantelli}. Lemma~\ref{lem:Borel_Cantelli_applied} will then imply that for a.e.\@~$x\in\{0,1\}^\N$ the distributions of $M_k^x$ and $M_k$ share the same asymptotic behaviour, thus passing the annealed result of Proposition~\ref{prop:convergence_annealed} to its corresponding quenched version, Theorem~\ref{thm:main_teo}.

\begin{proof}[Proof of Theorem~\ref{thm:main_teo}]
Let $p\in(1/2,1)$ and fix $n\ge0$. For $\beta(k)=\sqrt{k}\log k$, let $G_k(\beta)$ be as in~\eqref{eq:G_k_and_H_k}. The function $f_k(x)$ defined in~\eqref{eq:f_k} depends only on the first $m_k=N_k+k-1$ coordinates of~$x$. 
Moreover, changing a single coordinate of~$x$ can affect the count 
$\#\{\omega : M_k^x(\omega) = n\}$ by at most~$k$. Since for any~$\omega\in G_k$,
\[
\Ber^k(\omega)\le p^{pk+\beta(k)}(1-p)^{(1-p)k-\beta(k)}=
2^{-k\cdot H(p)}\prs{p/(1-p)}^{\beta(k)}=d_k,
\]
the inequality~\eqref{eq:inequality_single_coordinate} is satisfied for $f_k$ with $c=kd_k$. By Proposition~\ref{prop:McD},
\begin{equation}\label{eq:McD_applied}
\Ber^{\N}\Bigl(x:\bigl|f_k(x)-
{\textstyle\int f_k\diff\Ber^\N}\bigr|>1/k\Bigr)\le
2\exp\Bigl\{-2\bigl(k^4(N_k+k-1)d_k^{2}\bigr)^{-1}\Bigr\}.
\end{equation}

Assume that $N_k=O(2^{k\cdot 3H(p)/2})$. This asymptotic covers both cases~\ref{item:conv_0_quenched} and 
\ref{item:partial_escape_quenched}, and part of~\ref{item:vague_0_quenched}. Under this assumption, there is $C>0$ such that, for any sufficiently large $k\ge1$,
\begin{equation*}
1/(k^4\cdot d_k^2\cdot N_k)\ge C k^{-4}\cdot 2^{k\cdot H(p)/2} \prs{p/(1-p)}^{-2\beta(k)}.    
\end{equation*}
Since the latter grows exponentially fast, using the bound~\eqref{eq:McD_applied} we obtain condition~\eqref{eq:Borel_Cantelli}.
By Lemma~\ref{lem:Borel_Cantelli_applied}, 
\begin{equation*}
|\Ber^k(M^x_k=n)-P_k(M_k=n)|\to 0,
\end{equation*}
for a.e.\@~$x\in\{0,1\}^\N$. 
By Proposition~\ref{prop:convergence_annealed}, this concludes the proof for~\ref{item:conv_0_quenched},~\ref{item:partial_escape_quenched}, and part of~\ref{item:vague_0_quenched}.

We now show point~\ref{item:vague_0_quenched}, dropping the additional assumption on $N_k$.
Consider a new sequence 
\[\textstyle
\widehat N_k=\min\prs{N_k,\lfloor2^{k\cdot 3H(p)/2}\rfloor},
\]
and define from it a new $\widehat M_k^x$ as in~\eqref{eq:M_k^x}. By what seen above, $P_k(\widehat M_k^x\ge n)\to 1$  for a.e.\@~$x$. Since for every $x\in\{0,1\}^\N$ 
\[
P_k(M_k^x\ge n)\ge P_k(\widehat M_k^x\ge n)\to 1,
\]
the proof of~\ref{item:vague_0_quenched} is finished.

\end{proof}

We now have all the necessary tools to show that for $p\neq 1/2$ and a.e.\@ $x\in\{0,1\}^\N$, the sequence $M^x_k$ cannot converge in distribution to a Poisson random variable.
We recall that a sequence $Z_k$ of random variables converges in distribution to $\Po(\lambda)$ for some $\lambda>0$, if 
\[
\lim_{k\to\infty}\P(Z_k=n)=\frac{\lambda^n}{e^{\lambda}\cdot n!}>0.
\]
for any $n\ge0$.

\begin{proof}[Proof of Corollary~\ref{cor:non_Poisson}]
Let $p\neq 1/2$ and consider a sequence $N_k$ of positive integers. Without loss of generality we assume $p\in(1/2,1)$.  
For any $k\ge1$ we may derive a unique $a_k>0$ such that $N_k=2^{k\cdot H(p)}{a_k}^{\sqrt{k}}$. So, there exists a subsequence $a_{k_j}$ satisfying one of the following:
    \begin{enumerate}[(a)]
        \item\label{item:1} $\lim_{j\to\infty}a_{k_j}= 0$.
        \item\label{item:2} $\lim_{j\to\infty}a_{k_j}= \infty$.
        \item\label{item:3} $\lim_{j\to\infty}a_{k_j}= a$, for some $a>0$.
    \end{enumerate}
In cases~\ref{item:1} and~\ref{item:2}, the sequence $N_{k_j}$ satisfies respectively the assumptions of points~\ref{item:conv_0_quenched} and~\ref{item:vague_0_quenched} of Theorem~\ref{thm:main_teo}. In~\ref{item:3}, the asymptotic of the sequence $N_{k_j}$ is the same as point~\ref{item:partial_escape_quenched} of Theorem~\ref{thm:main_teo}. 
In all three cases, for $\Ber^\N$-a.e.\@ $x\in\{0,1\}^\N$ and any $n\ge1$, we have that $\Ber^k(M^x_{k_j}=n)\to 0$ as $j\to\infty$, thus disproving the convergence of~$M^x_k$ to any non-zero random variable.
\end{proof}

\begin{remark}
An alternative proof of Corollary~\ref{cor:non_Poisson} can be carried via the annealed result of Section~\ref{sec:annealed}.
Reasoning as in the proof of Corollary~\ref{cor:non_Poisson} and using Proposition~\ref{prop:convergence_annealed} in place of Theorem~\ref{thm:main_teo}, we can show that the sequence $M_k$ does not converge in distribution to Poisson. Since quenched implies annealed, the contrapositive argument yields the thesis of Corollary~\ref{cor:non_Poisson}. 
\end{remark}

We conclude our paper with the quenched version of the conditional Poisson convergence of Proposition~\ref{prop:conditional_poisson_annealed}. For  $c\in\R$ and $k\ge1$, we let $n_k=\lfloor pk-c\sqrt{k}\rfloor$ and $F_k$ be the set from~\eqref{eq:F_k}. Let $\nu^k=\Ber^k(\cdot\mid F_k)$, as defined in~\eqref{eq:nu}. We recall that $\ddot M_k\colon \Omega_k\to\N\cup\{0\}$ is a sequence defined as in~\eqref{eq:Mk}, according to the probability measure $\ddot P_k=\Ber^\N\times\nu^k$. For any $x\in\{0,1\}^\N$, we let $\ddot M_k^x=\ddot M_k(x,\cdot)$.
As we did in the proof of Theorem~\ref{thm:main_teo}, in the following we follow the ideas of~\cite{AlvBecMer23}, showing that $\ddot M_k$ and $\ddot M_k^x$ share the same asymptotic for a.e.\@~$x$.

\begin{proof}[Proof of Theorem~\ref{thm:quenched_conditional_poisson}]
Let $p\in(0,1)$, $\lambda>0$, $q=p^{n_k}(1-p)^{k-n_k}$ as defined in~\eqref{def:q}, and suppose that 
$N_k=\floor{\lambda/q}$.
Let $n\ge0$, and define for $k\ge1$ the function $g_k\colon\{0,1\}^\N\to[0,1]$ as
\begin{equation*}\label{eq:g_k}
g_{k}(x)=\nu^k(\omega\in\{0,1\}^k:\ddot M_k^x=n).    
\end{equation*} 
Then, $g_k(x)$ depends only on the first $m_k=N_k+k-1$ coordinates of~$x$. By~\eqref{eq:mu^k(omega)}, for any $\omega\in\{0,1\}^\N$
\[
\nu^k(\omega)=\Ber^{k}(\{\omega\}\cap F_k)/\Ber^k(F_k)\le 
q/\Ber^k(F_k)\le \lambda/(N_k\cdot \Ber^k(F_k))=e_k.
\]
So, the inequality~\eqref{eq:inequality_single_coordinate} is satisfied for $g_k$ with $c=ke_k$.
By Proposition~\ref{prop:McD},
\begin{equation*}
\Ber^{\N}\Bigl(x:\bigl|g_k(x)-
{\textstyle\int g_k\diff\Ber^\N}\bigr|>1/k\Bigr)\le
2\exp\Bigl\{-2\bigl(k^4(N_k+k-1)e_k^{2}\bigr)^{-1}\Bigr\}.
\end{equation*}
Since
\[
1/(k^4\cdot N_k\cdot e_k^2)=N_k (\Ber^k(F_k))^2/(\lambda k^4),
\]
where $N_k$ grows exponentially and $\Ber^k(F_k)=\Theta(1/\sqrt{k})$ by~\eqref{eq:local_CLT}, it follows that
\[
\sum_{k=1}^\infty \Ber^{\N}\Bigl(x:\bigl|g_k(x)-
{\textstyle\int g_k\diff\Ber^\N}\bigr|>1/k\Bigr)<\infty.
\]
By Borel-Cantelli, for $\Ber(p)^\N$-a.e.\@~$x$:
\begin{equation*}\textstyle
|g_k(x)-\int g_k\diff\Ber^\N|\to0.
\end{equation*}
Reasoning as in~\eqref{eq:Tonelli} and applying Tonelli's theorem, we get $\int g_k\diff\Ber^\N=\ddot P_k(\ddot M_k=n)$. The proof is finished by Proposition~\ref{prop:conditional_poisson_annealed}.
\end{proof}

\subsection*{Acknowledgments}
The authors thank Yuval Peled and Mike Hochman for helpful discussions about the paper.

\subsection*{Funding}

This work was supported in part by ISF grants 3464/24 (first author) and 3056/21 (second author).

\bibliographystyle{apalike.bst}
\bibliography{References}

\end{document}